\pdfoutput=1 
\documentclass[11pt]{amsart}
\usepackage[utf8]{inputenc}
\usepackage{amssymb,amsmath,amsthm,enumerate,enumitem,colonequals,mlmodern,tikz-cd,microtype}
\usepackage[cal=euler,bfcal,bb=px,bfbb]{mathalpha}
\usepackage[top=3.75cm, bottom=3cm, left=3.5cm, right=3.5cm]{geometry}

\makeatletter
\@namedef{subjclassname@2020}{\textup{2020} Mathematics Subject Classification}
\makeatother

\usepackage{xcolor}
\colorlet{darkblue}{blue!55!black}
\colorlet{darkcyan}{cyan!50!black}
\colorlet{darkgreen}{green!60!black}

\usepackage{hyperref}
\hypersetup{
    colorlinks=true,
    linkcolor= darkblue,
    urlcolor= darkcyan,
    citecolor= darkgreen,
}

\def\eqref#1{\textcolor{darkblue}{(\ref{#1})}}

\PassOptionsToPackage{hyphens}{url}
\usepackage{hyperref}

\usepackage[nameinlink]{cleveref} 
\Crefformat{section}{#2\S#1#3} 
\Crefmultiformat{section}{#2\S\S#1#3}{ and~#2#1#3}{, #2#1#3}{, and~#2#1#3}
\crefname{hypothesis}{hypothesis}{hypotheses}
\Crefname{hypothesis}{Hypothesis}{Hypotheses}

\usepackage[pagewise]{lineno}
\let\oldequation\equation
\let\oldendequation\endequation
\renewenvironment{equation}{\linenomathNonumbers\oldequation}{\oldendequation\endlinenomath}
\expandafter\let\expandafter\oldequationstar\csname equation*\endcsname
\expandafter\let\expandafter\oldendequationstar\csname endequation*\endcsname
\renewenvironment{equation*}{\linenomathNonumbers\oldequationstar}{\oldendequationstar\endlinenomath}
\let\oldalign\align
\let\oldendalign\endalign
\renewenvironment{align}{\linenomathNonumbers\oldalign}{\oldendalign\endlinenomath}
\expandafter\let\expandafter\oldalignstar\csname align*\endcsname
\expandafter\let\expandafter\oldendalignstar\csname endalign*\endcsname
\renewenvironment{align*}{\linenomathNonumbers\oldalignstar}{\oldendalignstar\endlinenomath}


\newcounter{intro}
\newcounter{HypCounter}

\newtheorem{introthm}[intro]{Theorem}

\newtheorem{introcor}[intro]{Corollary}

\theoremstyle{plain}
\newtheorem{theorem}{Theorem}[section]
\newtheorem{lemma}[theorem]{Lemma}
\newtheorem{corollary}[theorem]{Corollary}
\newtheorem{proposition}[theorem]{Proposition}

\theoremstyle{definition}

\newtheorem{definition}[theorem]{Definition}
\newtheorem{example}[theorem]{Example}
\newtheorem{nonexample}[theorem]{Non-Example}
\newtheorem{hypothesis}[HypCounter]{Hypothesis}
\newtheorem*{hypothesis*}{Hypothesis}

\newtheorem{remark}[theorem]{Remark}

\newtheorem*{disclaimer}{Disclaimer}
\newtheorem*{ack}{Acknowledgements}

\numberwithin{equation}{section}
\numberwithin{theorem}{section}

\title[Descending strong generation]{Descending strong generation \\ in algebraic geometry}

\author[T.~De Deyn]{Timothy De Deyn}
\address{T.~De Deyn,
School of Mathematics and Statistics,
University of Glasgow, 
Glasgow G12 8QQ,
United Kingdom}
\email{timothy.dedeyn@glasgow.ac.uk}

\author[P.~Lank]{Pat Lank}
\address{P.~Lank,
Dipartimento di Matematica “F. Enriques”, Universit\`{a} degli Studi di Milano, Via Cesare
Saldini 50, 20133 Milano, Italy}
\email{plankmathematics@gmail.com}

\author[K. ~Manali Rahul]{Kabeer Manali Rahul}
\address{K. ~Manali Rahul,
Center for Mathematics and its Applications, 
Mathematical Science Institute, Building 145, 
The Australian National University, 
Canberra, ACT 2601, Australia}
\email{kabeer.manalirahul@anu.edu.au}

\date{\today}



\begin{document}
    
\begin{abstract}
    We formalize the main approach for showing Zariski descent-type statements for strong generation of triangulated categories associated to algebro-geometric objects. 
    This recovers various known statements in the literature. 
    As applications we show that strong generation for the singularity category of a Noetherian separated scheme is Zariski local and obtain a strong generation result for the bounded derived category of a Noetherian concentrated algebraic stacks with finite diagonal.
\end{abstract}

\maketitle

\tableofcontents

\section{Introduction}
\label{sec:introduction}

The goal for this note is twofold.
Firstly, we show that \emph{(strong) generation} for the \emph{singularity category}, as  introduced in \cite{Buchweitz/Appendix:2021,Orlov:2004}, of a (Noetherian separated) scheme is a \emph{Zariski local} question.
Secondly, as the techniques involved in proving these types of statements are `standard' at this point \cite{Neeman:2021b,Aoki:2021,Lank:2024,DeDeyn/Lank/ManaliRahul:2024a,DeDeyn/Lank/ManaliRahul:2024b,Stevenson:2025} and are quite formal, it felt like a pleasant exercise in abstraction to make these explicit.
We do this below in our main result, \Cref{thm:generators_presheaf}, which is stated in the language of `presheaves of triangulated categories'.
Instead of unveiling the theorem here (and all the definitions that are needed to formulate the statement), we simply state the main corollaries. 

\begin{introcor}[\Cref{cor:descent_Dsg,cor:lank2024,cor:Neeman2021}]
    Let $X$ be a separated Noetherian scheme, suppose $\mathcal{T}(-)$ is either $\operatorname{Perf}(-)$ (:= the category of perfect complexes), $D^b_{\operatorname{coh}}(-)$ (:= the derived category of complexes with bounded and coherent cohomology) or $D_{\operatorname{sg}}(-)$ (:= the Verdier quotient of the latter by the former), and let $X=U\cup V$ be an open cover.
    Then $\mathcal{T}(X)$ admits a strong (resp.\ classical) generator if and only if both $\mathcal{T}(U)$ and $\mathcal{T}(V)$ admit strong (resp.\ classical) generators.
\end{introcor}

In addition, in \Cref{sec:variations} we discuss some variations: a `big' version of \Cref{thm:generators_presheaf}, how one can tackle descent for more `exotic' Grothendieck topologies than the Zariski topology, and we give an application to strong generation of algebraic stacks.
More precisely, we show that strong generation of the bounded derived category of a (separated Noetherian concentrated) algebraic stack can be checked along a `nice enough' cover, see \Cref{subsec:stacks} for an explanation of some of the terminology in the next statement (but note that a tame stack is always concentrated). 

\begin{introthm}[{\Cref{thm:quasi-DM_strong_generation}}]
    Let $\mathcal{X}$ be a separated Noetherian concentrated algebraic stack.
    Assume there exists a surjective, separated, finitely presented, quasi-finite flat morphism $U\to \mathcal{X}$ from a scheme such that $D^b_{\operatorname{coh}}(U)$ admits a strong (resp.\ classical) generator. 
    Then $D^b_{\operatorname{coh}}(\mathcal{X})$ admits a strong (resp.\ classical) generator.
\end{introthm}

Let us briefly spend a few words emphasizing the main ideas/tools behind the proof of \Cref{thm:generators_presheaf}.
To formalize Zariski descent of strong generation we took inspiration from \cite{Hall/Rydh:2017}, and prove descent of strong generation for presheaves of triangulated categories along suitably behaved Cartesian squares. 
The two main puzzle pieces are
\begin{itemize}
    \item the existence of Mayer--Vietoris triangles, which allows one relate the global to the local and
    \item a way to control the `local terms' in the previous triangles.
\end{itemize}
The latter is achieved by requiring certain functors associated to the presheaf to be `(strongly) compactly or coherently bounded', see \Cref{sec:boundedness}.
This formalizes the key ingredients in the previous proofs in the literature.
All in all, the above allows one to use local (strong) generators to obtain global ones.

That compact generation can be checked locally (in more than just the Zariski topology) is a well-established fact \cite{Neeman:1996,Rouquier:2008,Toen:2012,Hall/Rydh:2017}.
It is only relatively recently that strong (or classical) generation of the bounded or perfect derived category has been studied in a like manner.
In fact, it is the (strong) compact/coherent boundedness which is crucial for proving \emph{strong} generation statements.

Lastly, we should mention that our choice of using presheaves of triangulated categories to formalize Zariski descent was somewhat arbitrary---and mainly chosen for the algebro-geometric biases in our minds. 
The approach with cocovering taken in \cite{Rouquier:2008} or with formal Mayer--Vietoris squares in \cite{Balmer/Favi:2007} are equally valid ways of proceeding (except for the stacky part in \Cref{subsec:stacks}).

The note proceeds as follows.
In \Cref{sec:recol} we recall some preliminary definitions after which the `presheaves of triangulated categories' are introduced in \Cref{sec:presheaf_tricats}.
Next in \Cref{sec:boundedness} we discuss the various crucial boundedness conditions needed to formalize Zariski descent in \Cref{sec:Zariski}. 
Lastly, in \Cref{sec:variations}, we discuss some generalizations: `big' versions, how to deal with more exotic topologies and an application to strong generation of algebraic stacks.

\begin{disclaimer}
    For simplicity we added a Noetherianity assumption in some places.
    In many cases this can be removed and one only requires quasi-compactness and (quasi-)separatedness; either assuming the scheme (or stack) is coherent (i.e.\ has coherent structure sheaf), or by  looking at the bounded pseudo-coherent objects instead of the bounded derived category of complexes with coherent cohomology.
    
\end{disclaimer}

\begin{ack}
    Timothy De Deyn was supported by ERC Consolidator Grant 101001227 (MMiMMa). Pat Lank and Kabeer Manali Rahul were supported by ERC Advanced Grant 101095900-TriCatApp. Kabeer Manali Rahul was also supported by an Australian Government Research Training Program Scholarship. 
    The authors would like to thank both Jack Hall and Fei Peng for helpful discussions concerning algebraic stacks and Greg Stevenson for pointing out the reference \cite{Simson:1977}.
\end{ack}

\section{Triangular recollections}
\label{sec:recol}

We briefly recall a few concepts needed below, and more importantly we fix notation used throughout. 
In this section, $\mathsf{T}$ denotes a triangulated category with shift functor $[1]$.

We begin by recalling some notation and definitions related to generation for triangulated categories introduced in \cite{Bondal/VandenBergh:2003, Rouquier:2008}.
Let $\mathsf{S}\subseteq\mathsf{T}$ be a (not-necessarily triangulated) full subcategory and let\footnote{The conventions for what is meant by $\operatorname{add}$ are not uniform in the literature \cite{Bondal/VandenBergh:2003,Rouquier:2008,Neeman:2021b} and our convention mixes this a bit more. 
The primary difference is whether or not one allows for $\operatorname{add}$ to be closed under direct summands and/or shifts.
A similar remark holds for $\operatorname{Add}$ defined below.
} $\operatorname{add}\mathsf{S}$ denote the smallest (strictly) full subcategory of $\mathsf{T}$ containing $\mathsf{S}$ closed under shifts, finite coproducts and direct summands.
Inductively, define 
\begin{displaymath}
    \langle\mathsf{S}\rangle_n :=
    \begin{cases}
        \operatorname{add}\varnothing & \text{if }n=0, \\
        \operatorname{add}\mathsf{S} & \text{if }n=1, \\
        \operatorname{add}\{ \operatorname{cone}\phi \mid \phi \in \operatorname{Hom}(\langle \mathsf{S} \rangle_{n-1}, \langle \mathsf{S} \rangle_1) \} & \text{if }n>1.
    \end{cases}
\end{displaymath}
Moreover, let 
\begin{displaymath}
\langle\mathsf{S}\rangle:=\langle\mathsf{S}\rangle_\infty:=\bigcup_{n\geq 0} \langle\mathsf{S}\rangle_n.
\end{displaymath}
This is the smallest \textbf{thick}, i.e.\ closed under direct summands, triangulated subcategory containing $\mathsf{S}$.
Similarly, if $\mathsf{T}$ admits small coproducts, define $\operatorname{Add}\mathsf{S}$ exactly as $\operatorname{add}\mathsf{S}$ but additionally closed under arbitrary small coproducts, and $\overline{\langle \mathsf{S}\rangle}_n$ as above, replacing $\operatorname{add}$ by $\operatorname{Add}$ everywhere.
Furthermore, when $\mathsf{S}=\{ G\}$ consists of a single object, we leave out the brackets, simply writing $\langle G\rangle$, etc.

We say an object $G\in\mathsf{T}$ is a \textbf{classical generator} if $\langle G\rangle = \mathsf{T}$, a \textbf{strong generator} if $\langle G\rangle_{n+1} = \mathsf{T}$ for some $n\geq 0$ and a \textbf{strong $\oplus$-generator} if $\overline{\langle G\rangle}_{n+1} =\mathsf{T}$ for some $n\geq 0$. 
The \textbf{Rouquier dimension} of $\mathsf{T}$, denoted $\dim \mathsf{T}$, is the smallest integer $n$ such that $\mathsf{T} = \langle G \rangle_{n+1}$ for some object $G$ in $\mathsf{T}$, and is $\infty$ when no such objects exists. 

Next, recall from \cite{Beilinson/Berstein/Deligne/Gabber:2018} that pair of strictly full subcategories $\tau=(\mathsf{T}^{\leq 0},\mathsf{T}^{\geq 0})$ of $\mathsf{T}$ is called a \textbf{$t$-structure} if the following conditions are satisfied:
\begin{enumerate}
    \item $\operatorname{Hom}(\mathsf{T}^{\leq 0},\mathsf{T}^{\geq 0}[-1])=0$,
    \item $\mathsf{T}^{\leq 0}[1]\subseteq \mathsf{T}^{\leq 0}$ and $\mathsf{T}^{\geq 0}[-1]\subseteq \mathsf{T}^{\geq 0}$,
    \item for any $E\in\mathsf{T}$, there exists a distinguished triangle
    \begin{displaymath}
        A \to E \to B \to A[1]
    \end{displaymath}
    with $A\in\mathsf{T}^{\leq 0}$ and $B\in\mathsf{T}^{\geq 0}[-1]$.
\end{enumerate}

Suppose $(\mathsf{T}^{\leq 0}, \mathsf{T}^{\geq 0})$ is a $t$-structure.
For any integer $n$ define $\mathsf{T}^{\leq n}:= \mathsf{T}^{\leq 0}[-n]$ and $\mathsf{T}^{\geq n}:= \mathsf{T}^{\geq 0}[-n]$.
The pair $(\mathsf{T}^{\leq n}, \mathsf{T}^{\geq n})$ is also a $t$-structure on $\mathsf{T}$.
Moreover, we use the following notation
\begin{displaymath}
    \mathsf{T}^{-}:= \bigcup_{n=0}^\infty \mathsf{T}^{\leq n},\quad \mathsf{T}^{+} := \bigcup^\infty_{n=0} \mathsf{T}^{\geq -n},\quad \mathsf{T}^b := \mathsf{T}^{-} \bigcap \mathsf{T}^{+}.
\end{displaymath}

Furthermore, suppose $\mathsf{T}$ admits small coproducts and let $\mathsf{T}^c$ denote the full subcategory of \textbf{compact objects}, i.e.\ those $E\in\mathsf{T}$ for which $\operatorname{Hom}(E,-)$ commutes with coproducts.
We denote by $\mathsf{T}_c^-$ the full subcategory consisting of those objects $E\in\mathsf{T}$ which admit, for every integer $n>0$, a distinguished triangle $D\to E \to F \to E[1]$ with $D\in\mathsf{T}^c$ and $F\in\mathsf{T}^{\leq -n}$. Further, put $\mathsf{T}^b_c:=\mathsf{T}_c^- \cap \mathsf{T}^b$.
Lastly, recall that a triangulated functor $F:\mathsf{T}\to \mathsf{S}$ between triangulated categories equipped with $t$-structures is called \textbf{$t$-exact} if 
$F(\mathsf{T}^{\leq 0})\subseteq \mathsf{S}^{\leq 0}$ and $F(\mathsf{T}^{\geq 0})\subseteq \mathsf{S}^{\geq 0}$.

\section{Presheaves of triangulated categories}
\label{sec:presheaf_tricats}

It is time that the main protagonists enter the stage. 
These are the `presheaves of triangulated categories'---often further along equipped with extra structures.
Denote the $2$-category of triangulated categories by $\operatorname{t\mathbf{Cat}}$.

\begin{definition}
    Let $\mathcal{C}$ be a small category (usually admitting all finite limits).
    A \textbf{$\mathcal{C}$-presheaf of triangulated categories} is a 2-functor\footnote{Throughout, we use $\mathcal{T}$ to refer to a presheaf of triangulate categories and $\mathsf{T}$ to refer to a single one.}
    \begin{displaymath}
        \mathcal{T}\colon \mathcal{C}^{\operatorname{op}} \to \operatorname{t\mathbf{Cat}}.
    \end{displaymath}
    For any morphism $f\colon U \to V$ in $\mathcal{C}$, there is an induced triangulated functor $f^\ast_\mathcal{T} \colon \mathcal{T}(V) \to \mathcal{T}(U)$ referred to as the \textbf{pullback along $f$}. 
    Moreover, define the full subcategory
    \begin{displaymath}
        \mathcal{T}_{V\setminus U}(V):=\ker(f^\ast_\mathcal{T})=\{ F\in \mathcal{T}(V)\mid f^\ast_\mathcal{T}F=0\} 
    \end{displaymath}
    to be the kernel of $f^\ast_\mathcal{T}$.
\end{definition}

When no confusion can occur we omit the subscripts in notation for the pullbacks of a presheaf of triangulated categories.

\begin{example}\label{ex:presheaves_on_Zariski_site}
    There are many examples to be found in algebraic geometry.
    Before giving the easiest, let us recall some common notation: 
    For a scheme $X$, $D_{\operatorname{qc}}(X)$ denotes the derived category of complexes of $\mathcal{O}_X$-modules with quasi-coherent cohomology, $D^b_{\operatorname{coh}}(X)$ denotes the subcategory of those complexes with bounded and coherent cohomology and $\operatorname{Perf}(X)$ denotes the subcategory consisting of perfect complexes (which equals $D_{\operatorname{qc}}(X)^c$ when $X$ is quasi-compact and quasi-separated).
    Lastly, let $D_{\operatorname{sg}}(X)$ denote the Verdier quotient $D^b_{\operatorname{coh}}(X)/\operatorname{Perf}(X)$. 
    
    Consider $X_{Zar}$, the small Zariski site of a scheme $X$  (see \cite[\href{https://stacks.math.columbia.edu/tag/020T}{Tag 020T}]{stacks-project}). 
    It is the category with as objects the open subsets of $X$ and as morphisms the inclusions.    
    The following define presheaves of triangulated categories on $X_{Zar}$ (we are using that open immersions are flat):
    \begin{displaymath}
    D_{\operatorname{qc}},\quad D^b_{\operatorname{coh}},\quad \operatorname{Perf}, \quad\text{and}\quad D_{\operatorname{sg}}.
    \end{displaymath}
    These are the natural examples to have in mind for formalizing Zariski descent for generation.
\end{example}

\begin{definition}\label{def:compactly_generated_presheaf}
    Let $\mathcal{T}$ be a $\mathcal{C}$-presheaf of triangulated categories.    \begin{enumerate}
        \item We say $\mathcal{T}$ admits \textbf{adjoints} if for every $f\colon U \to V$ in $\mathcal{C}$, the pullback $f^\ast_\mathcal{T} \colon \mathcal{T}(V) \to \mathcal{T}(U)$ admits a right adjoint $f_{\mathcal{T},\ast} \colon \mathcal{T}(U) \to \mathcal{T}(V)$ (referred to as the \textbf{pushforward along $f$}).
        \item We say $\mathcal{T}$ is \textbf{compactly generated} if $\mathcal{T}(U)$ is compactly generated for all $U$ in $\mathcal{C}$ (and so by definition admits all small coproducts) and the pullback functors commute with small coproducts.
    \end{enumerate}
\end{definition}

Observe that it immediately follows by Brown representability \cite[Theorem 4.1]{Neeman:1996} that a compactly generated presheaf admits adjoints (as we require the pullback functors to commute with coproducts). 

Let us endow the presheaves with more structure.

\begin{definition}
    Let $\mathcal{T}$ be a $\mathcal{C}$-presheaf of triangulated categories.
    A \textbf{$t$-structure on $\mathcal{T}$} consists of a pair of subpresheaves (defined in the obvious way) $\tau=(\mathcal{T}^{\leq 0},\mathcal{T}^{\geq 0})$ of $\mathcal{T}$ such that the pair $\tau_U=(\mathcal{T}^{\leq 0}(U),\mathcal{T}^{\geq 0}(U))$ is a $t$-structure on $\mathcal{T}(U)$ for each object $U$ of $\mathcal{C}$.
\end{definition}

It's time to introduce the first of the puzzle: the `Mayer--Vietoris triangles' (see \Cref{lem:MVtriangle} below).
There are numerous ways of setting this up, we follow the approach of \cite{Hall/Rydh:2017} as this is more in-line with the presheaf approach taken here, but other approaches exist.
In particular, those of Balmer--Favi \cite{Balmer/Favi:2007} and Rouquier \cite{Rouquier:2008} are equally valid, see \Cref{rmk:cocoverings} where we make the link.

To this end, we first recall the notion of `flatness' as defined in \cite[Definition 5.1]{Hall/Rydh:2017}. 

\begin{definition}
    Let $\mathcal{T} \colon \mathcal{C}^{\operatorname{op}} \to \operatorname{t\mathbf{Cat}}$ be a presheaf of triangulated categories with adjoints. 
    A morphism $f\colon U \to V$ in $\mathcal{C}$ is called \textbf{$\mathcal{T}$-preflat} if for every Cartesian square (in $\mathcal{C}$)
    \begin{displaymath}
        \begin{tikzcd}[ampersand replacement=\&]
            {U_W} \& {W} \\
            U \& V\rlap{ ,}
            \arrow["{f_W}", from=1-1, to=1-2]
            \arrow["{g_U}"', from=1-1, to=2-1]
            \arrow["g", from=1-2, to=2-2]
            \arrow["f"', from=2-1, to=2-2]
        \end{tikzcd}
    \end{displaymath}
    the canonical base change map $f^\ast g_\ast \to (g_U)_\ast (f_W)^\ast$ is an isomorphism.
    If, additionally, for every morphism $V^\prime \to V$ the pullback $f^\prime \colon V \to V^\prime$ of $f$ is $\mathcal{T}$-preflat, we say the morphism $f$ is \textbf{$\mathcal{T}$-flat}. 
\end{definition}

\begin{example}\label{ex:T-flatness_in_AG}
    Suppose $\mathcal{C}=\operatorname{Sch}$ is the category of schemes (and we silently agree to ignore the fact that it is not small), and view the presheaf $D_{\operatorname{qc}}$ on it.
    As noted in \cite[Example 5.2]{Hall/Rydh:2017} a morphism $f\colon X\to Y$ of schemes is $D_{\operatorname{qc}}$-flat if and only if it flat.
    In particular, if we restrict to locally finitely presented morphisms, the $D_{\operatorname{qc}}$-flat monomorphisms are exactly the open immersions.
\end{example}

In fact, essentially, only $\mathcal{T}$-flat monomorphism will appear in the rest of this work.
As the above example shows, these correspond exactly to open immersions in algebraic geometry and so help formalize `Zariski descent'.
In fact, essentially, all we need is that they give us localizations after pulling back.

\begin{lemma}\label{lem:preflatmono_are_loc}
    Let $\mathcal{T}$ be a $\mathcal{C}$-presheaf of triangulated categories with adjoints. If  $i\colon U\to X$ is a $\mathcal{T}$-preflat monomorphism in $\mathcal{C}$, then the counit of adjunction $i^\ast i_\ast \to 1$ is an isomorphism.
    In particular, $i^*$ is a localization, i.e.\ we have a localization sequence
    \begin{displaymath}
        \mathcal{T}_{X\setminus U}(X)\longrightarrow \mathcal{T}(X)\xrightarrow{\  i^*\ } \mathcal{T}(U).
    \end{displaymath}
\end{lemma}

\begin{proof}
    This is \cite[Lemma 5.4]{Hall/Rydh:2017}, we include the its short proof for convenience of the reader. 
    As $i$ is a monomorphism the square
    \begin{displaymath}
        \begin{tikzcd}[ampersand replacement=\&, sep=1.5em ]
            {U} \& {U} \\
            U \& X
            \arrow["{1}", from=1-1, to=1-2]
            \arrow["{1}"', from=1-1, to=2-1]
            \arrow["i", from=1-2, to=2-2]
            \arrow["i"', from=2-1, to=2-2]
        \end{tikzcd}
    \end{displaymath}
    is Cartesian, so by preflatness, $j^\ast j_\ast = 1_{\ast} 1^\ast =1$.    
\end{proof}

The following is a more restrictive variation of \cite[Definition 5.5]{Hall/Rydh:2017} that suits our purposes, see also \Cref{rmk:MVsquare} below. 

\begin{definition}\label{def:ZMV_square}
    Let $\mathcal{T}$ be a $\mathcal{C}$-presheaf of triangulated categories with adjoints. A \textbf{Zariski $\mathcal{T}$-square} is a Cartesian square (in $\mathcal{C}$)
    \begin{displaymath}
        \begin{tikzcd}[ampersand replacement=\&]
            {W} \& {V} \\
            U \& X
            \arrow[from=1-1, to=1-2]
            \arrow[from=1-1, to=2-1]
            \arrow["j", from=1-2, to=2-2]
            \arrow["i"', from=2-1, to=2-2]
        \end{tikzcd}
    \end{displaymath}
    satisfying:
    \begin{enumerate}
        \item\label{item:ZMV_square1} the morphisms $i$ and $j$ are $\mathcal{T}$-flat monomorphism,
        \item\label{item:ZMV_square2} the intersection $\mathcal{T}_{X\setminus U}(X)\cap \mathcal{T}_{X\setminus V}(X) = 0$.
    \end{enumerate}
\end{definition}

\begin{remark}\label{rmk:MVsquare}
    In \Cref{def:ZMV_square} \eqref{item:ZMV_square1} we added the assumption that \emph{both} $i$ and $j$ are $\mathcal{T}$-monomorphism instead of requiring only one of them to be as in \cite{Hall/Rydh:2017}.
    Moreover, the condition, in loc.\ cit., that $j^\ast \colon \mathcal{T}_{X\setminus U}(X) \to \mathcal{T}_{V \setminus W}(V)$ is an equivalence was replaced by our condition \eqref{item:ZMV_square2}.
    It follows from \Cref{lem:preflatmono_are_loc} that, under the assumptions on $i$ and $j$, these are equivalent.
    Morally, c.f.\ \Cref{ex:T-flatness_in_AG}, we are restricting \emph{both} $i$ and $j$ to being open immersions with `$X=U\cup V$' (and therefore these squares are `only' capturing Zarsiki coverings).
    This is significantly more restrictive than \cite[Definition 5.5]{Hall/Rydh:2017}, which is why we simply call them `Zariski squares'.
\end{remark}

\begin{remark}\label{rmk:cocoverings}
    Our \Cref{def:ZMV_square} implies that $\{\mathcal{T}_{X\setminus U}(X),\allowbreak \mathcal{T}_{X\setminus V}(X) \}$ is a cocovering of $\mathcal{T}(X)$ in the sense of \cite[\S 5.3.3]{Rouquier:2008}; that the subcategories properly intersect follows from $\mathcal{T}$-(pre)flatness and \cite[Lemma 5.7. 2]{Rouquier:2008} (that these subcategories are `Bousfield' follows from \Cref{lem:preflatmono_are_loc}).
    Moreover, $\{\mathcal{T}_{X\setminus U}(X),\allowbreak \mathcal{T}_{X\setminus V}(X) \}$ also gives a formal Mayer--Vietoris square of $\mathcal{T}(X)$ in the sense of \cite[Definition 2.1]{Balmer/Favi:2007} (this follows from the intersection being zero and them intersecting properly). 
\end{remark}

We finish this section by showing that the above notion yields the familiar Mayer--Vietoris triangles.

\begin{lemma}\label{lem:MVtriangle}
    Let $\mathcal{T}$ be a $\mathcal{C}$-presheaf of triangulated categories with adjoints. Given a Zariski $\mathcal{T}$-square in $\mathcal{C}$
    \begin{displaymath}
        \begin{tikzcd}
            {W} & V \\
            U & X
            \arrow[from=1-1, to=1-2]
            \arrow[from=1-1, to=2-1]
            \arrow["j", from=1-2, to=2-2]
            \arrow["i"', from=2-1, to=2-2]
            \arrow["k"{description}, from=1-1, to=2-2]
        \end{tikzcd}
    \end{displaymath}
    there exists, for any $N\in\mathcal{T}(X)$, a (functorial) distinguished triangle
    \begin{equation}\label{eq:MVtriangle}
        N \to  
            i_\ast i^\ast N \oplus j_\ast j^\ast N \to   
            k_\ast k^\ast N \to N[1].
    \end{equation}
\end{lemma}

\begin{proof}
    Taking \Cref{rmk:cocoverings} into account there are, at least, three ways of showing this. 
    Either \cite[5.9 (1)]{Hall/Rydh:2017}, \cite[Proposition 5.10]{Rouquier:2008} or \cite[Lemma 2.13]{Balmer/Favi:2007} do the trick.
\end{proof}

\section{Boundedness conditions}
\label{sec:boundedness}

Next, comes the second crucial piece of the puzzle for making the arguments work; it allows us to `bound' the image of objects in a suitable fashion.

\begin{definition}\label{def:compact_boundedness}
    Let $F \colon \mathsf{T} \to \mathsf{S}$ be an exact functor between triangulated categories admitting small coproducts (usually compactly generated).
    We say
    \begin{enumerate}
        \item $F$ is \textbf{compactly bounded} if for each compact object $G\in\mathsf{T}^c$ there exits a compact object $H\in\mathsf{S}^c$ and an integer $n=n(G)\geq 0$ with $F(G)\in\overline{\langle H \rangle}_n$,
        \item $F$ is \textbf{strongly compactly bounded} if there exists an integer $n\geq0$ such that for each compact object $G\in\mathsf{T}^c$ there exits a compact object $H\in\mathsf{S}^c$ with $F(G)\in \overline{\langle H \rangle}_n$,
    \end{enumerate}
    and, if $\mathsf{T}$ and $\mathsf{S}$ are additionally equipped with $t$-structures,
    \begin{enumerate}[resume]
        \item $F$ is \textbf{coherently bounded} if for each object $G\in\mathsf{T}^b_c$ there exits an object $H\in\mathsf{S}^b_c$ and an integer $n\geq 0$ with $F(G)\in\overline{\langle H \rangle}_n$.
    \end{enumerate}
\end{definition}

It is worthwhile to observe that the conditions in \Cref{def:compact_boundedness} are closed under composition, we use this below without mention.

\begin{example}\label{ex:of_boundedness}
    Let $X$ be quasi-compact separated scheme and $i\colon U\to X$ an open immersion with $U$ quasi-compact.
    It is shown in \cite[Theorem 6.2]{Neeman:2021b} that $\mathbb{R}i_\ast$ (as functor on $D_{\operatorname{qc}}$) is compactly bounded.
    Moreover, using the tensor structure, one can show it is even strongly compactly bounded.
    Indeed, by compact boundedness of the pushforward there exists a perfect complex $P$ over $X$ and an integer $n\geq 0$ with $\mathbb{R}i_\ast \mathcal{O}_U\in\overline{\langle P \rangle}_n$. 
    Now, for any $G\in\operatorname{Perf}(U)$, there exists a $G^\prime\in\operatorname{Perf}(X)$ with $G$ a direct summand of $i^\ast G^\prime$, see e.g.\ \Cref{ex:presheaves_on_Zariski_site_satisfies_d} below.
    Consequently, tensoring $\mathbb{R}i_\ast \mathcal{O}_U$ with
    $G^\prime$ and applying the projection formula (see e.g.\ \cite[\href{https://stacks.math.columbia.edu/tag/08EU}{Tag 08EU}]{stacks-project}), one has $\mathbb{R}i_\ast G\in\overline{\langle P\otimes^{\mathbb{L}}_{\mathcal{O}_X} G^\prime \rangle}_n$. 
    As $H:=P\otimes^{\mathbb{L}}_{\mathcal{O}_X} G^\prime \in\operatorname{Perf}(X)$ and the integer $n$ does \emph{not} depend on $G$, and hence the claim follows. 
    Another class of strongly compactly bounded functores are, of course, given by morphisms whose pushforwards preserve perfect complexes, e.g.\ finite flat morphisms.

    Concerning coherently bounded morphisms, it is shown in \cite[Theorem 3.11]{Aoki:2021} that any morphism of finite type between separated Noetherian schemes is coherently bounded.
\end{example}

\begin{nonexample}
    One cannot expect the pushforward of any morphism of finite type between, e.g.\ Noetherian, schemes to be (strongly) compactly bounded.

    Indeed, let $f\colon Y\to X$ be a proper surjective morphism from a regular variety to a singular variety (e.g.\ $f$ could be a resolution of singularities or an alteration of $X$).
    Let $G$ be a classical generator of $\operatorname{Perf}(Y)$, this generator is strong by regularity of $Y$.
    Then, by \cite[Theorem 5.1]{Aoki:2021} or \cite[Theorem C]{Dey/Lank:2024}, $\mathbb{R}f_\ast G$ is a strong generator for $D^b_{\operatorname{coh}}(X)$. 
    Suppose that $\mathbb{R}f_\ast$ is compactly bounded.
    Then there would exist a perfect complex $H$ and integer $n\geq 0$ with $\mathbb{R}f_\ast G\in\overline{\langle H \rangle}_n$. 
    Consequently, by e.g.\ \cite[Lemma 2.14]{DeDeyn/Lank/ManaliRahul:2024a}, $D^b_{\operatorname{coh}}(X)=\langle H \rangle_n$ (potentially after increasing $n$) and so this also equals $\operatorname{Perf}(X)$ (as $H$ was perfect).
    However, this implies $X$ is regular 
    which is clearly a contradiction.
\end{nonexample}

In order to give more examples of (strongly) compactly bounded functors, we show that this condition can be checked `Zariski locally'.
In particular, this implies that the pushforward of any smooth morphism of quasi-compact separated schemes is compactly bounded (see \Cref{prop:smooth_compact_bounded} below).

From now on we use the following abuse of terminology. 
For a $\mathcal{C}$-presheaf $\mathcal{T}$ with adjoints and a morphism $f\colon X\to Y$ in $\mathcal{C}$, we say $f$ is compactly/coherently bounded when $f_\ast$ is.

\begin{proposition}[Local on target]\label{prop:comp_bound_local_target}
    Let $\mathcal{T}$ be a compactly generated $\mathcal{C}$-presheaf whose pullbacks preserve compacts.
    Assume, further, that $\mathcal{C}$ admits (non-empty) finite limits and $\mathcal{T}$-flat monomorphisms are compactly bounded\footnote{One can weaken this, of course. One only needs to require the necessary Cartesian squares to exist and the $\mathcal{T}$-flat monomorphisms that appear be compactly bounded.}. 
    Given a Zariski $\mathcal{T}$-square
            \begin{displaymath}
                \begin{tikzcd}
                    {W} & V \\
                    U & Y
                    \arrow[from=1-1, to=1-2]
                    \arrow[from=1-1, to=2-1]
                    \arrow["j", from=1-2, to=2-2]
                    \arrow["i"', from=2-1, to=2-2]
                    \arrow["k"{description}, from=1-1, to=2-2]
                \end{tikzcd}
            \end{displaymath}
    and a morphism $f \colon X \to Y$ in $\mathcal{C}$.
    Consider the following Cartesian squares
    \[
        \begin{tikzcd}[ampersand replacement=\&]
        {X_U} \& X \\
        U \& Y
        \arrow["{i_X}", from=1-1, to=1-2]
        \arrow["{f_U}"', from=1-1, to=2-1]
        \arrow["f", from=1-2, to=2-2]
        \arrow["i"', from=2-1, to=2-2]
        \end{tikzcd}\qquad\text{and}\qquad
        \begin{tikzcd}[ampersand replacement=\&]
        {X_V} \& X \\
        V \& Y\rlap{ .}
        \arrow["{j_X}", from=1-1, to=1-2]
        \arrow["{f_V}"', from=1-1, to=2-1]
        \arrow["f", from=1-2, to=2-2]
        \arrow["j"', from=2-1, to=2-2]
        \end{tikzcd}
    \]
    Then $f$ is (strongly) compactly bounded if and only if $f_U$ and $f_V$ are.
\end{proposition}

\begin{proof}    
    First, assume $f$ is compactly bounded.
    As $i\circ f_U=f\circ i_X$ is (strongly) compactly bounded and $i^\ast i_\ast =1$ (and pullbacks preserve compacts), it follows that $f_U$ is compactly bounded.
    Similar reasoning applies to $f_V$.

    Conversely, assume $f_U$ and $f_V$ are (strongly) compactly bounded and consider the commutative diagram whose squares are Cartesian
    \[
        \begin{tikzcd}[ampersand replacement=\&]
            {X_W} \& X_U  \& X \\
            W \& U \& Y\rlap{ .}
            \arrow["{}", from=1-1, to=1-2]
            \arrow["{f_W}"', from=1-1, to=2-1]
            \arrow["f_U", from=1-2, to=2-2]
            \arrow["", from=2-1, to=2-2]
            \arrow["{i_X}"', from=1-2, to=1-3]
            \arrow["{f}", from=1-3, to=2-3]
            \arrow["i", from=2-2, to=2-3]
            \arrow["k"', bend right=15, from=2-1, to=2-3]
            \arrow["k_X", bend left=15, from=1-1, to=1-3]
        \end{tikzcd}
    \]
    By the previous paragraph $f_W$ is (strongly) compactly bounded.
    Take  $G\in\mathcal{T}^c(X)$ arbitrary. 
    By $\mathcal{T}$-(pre)flatness of $i$, $j$ and $k$ together with \Cref{lem:MVtriangle} we have the following distinguished triangle
    \begin{equation}\label{eq:comp_bound_local_target}
        f_\ast G \to (i \circ f_{U})_\ast i_X^\ast G \oplus (j \circ f_{V})_\ast j_X^\ast G \to (k \circ f_{W})_\ast k_X^\ast G \to f_\ast G[1].
    \end{equation}
    By assumption $i \circ f_U$, $j\circ f_V$ and $k\circ f_W$ are (strongly) compactly bounded. 
    Consequently, we can bound the pushforwards of $i_X^\ast G$, $j_X^\ast G$ and $k_X^\ast G$; and, hence, also of $f^\ast G$ using the triangle \eqref{eq:comp_bound_local_target}.  
\end{proof}

\begin{proposition}[Local on source]\label{prop:comp_bound_local_source}
    Let $\mathcal{T}$ be a compactly generated $\mathcal{C}$-presheaf whose pullbacks preserve compacts.
    Assume, further, that $\mathcal{T}$-flat monomorphisms are compactly bounded\footnote{Again, one can weaker this, only requiring the $\mathcal{T}$-flat monomorphisms that appear be compactly bounded.}. 
    Given a Zariski $\mathcal{T}$-square
            \begin{displaymath}
                \begin{tikzcd}
                    {W} & V \\
                    U & X
                    \arrow[from=1-1, to=1-2]
                    \arrow[from=1-1, to=2-1]
                    \arrow["j", from=1-2, to=2-2]
                    \arrow["i"', from=2-1, to=2-2]
                    \arrow["k"{description}, from=1-1, to=2-2]
                \end{tikzcd}
            \end{displaymath}
    and a morphism $f \colon X \to Y$ in $\mathcal{C}$.
    Then $f$ is (strongly) compactly bounded if and only if $f \circ i$ and $f \circ j$ are.
\end{proposition}

\begin{proof}
    One direction is trivial as (strongly) compact boundedness is closed under composition.
    
    For the other direction, note that by the previous paragraph $f\circ k$ is (strongly) compactly bounded.
    Thus, applying $f_\ast$ to the distinguished triangle \eqref{eq:MVtriangle} obtained from \Cref{lem:MVtriangle}, gives a triangle one can use to bound the pushforward of $f_\ast$.
\end{proof}

\begin{corollary}\label{cor:comp_bound_local_schemes}
    Let $f \colon X \to Y$ be a morphism of quasi-compact separated schemes, $X = U^\prime \cup V^\prime$ and $Y=U\cup V$ be covers of quasi-compact opens.
    Then the following are equivalent
    \begin{enumerate}
        \item $f$ is (strongly) compactly bounded,
        \item\label{item:local_schemes_target} the restrictions of $f$ on the target, $f_U\colon f^{-1}(U)\to U$ and $f_V\colon f^{-1}(V)\to V$, are (strongly) compactly bounded,
        \item\label{item:local_schemes_source} the restrictions of $f$ on the source, $f|_{U^\prime}\colon U^\prime \to Y$ and $f|_{V^\prime}\colon V^\prime\to Y$, are (strongly) compactly bounded.
    \end{enumerate}
\end{corollary}

\begin{proof}
    This follows immediately from \Cref{prop:comp_bound_local_target,prop:comp_bound_local_source}, noting that by \Cref{ex:of_boundedness} pushing forward along quasi-compact open immersions (to quasi-compact separated schemes) is (strongly) compactly bounded.
\end{proof}

\begin{corollary}\label{cor:comp_bound_affine_local}
    Let $f\colon X \to Y$ be a morphism of quasi-compact separated schemes. 
    Suppose there exist affine open covers $Y=\cup_{i} U_i$ and $f^{-1}(U_i)=\cup_{j} V_j$ such that the restricted morphisms $V_j\to U_i$ are (strongly) compactly bounded, then $f$ is (strongly) compactly bounded.
\end{corollary}

\begin{proof}
    By quasi-compactness we can assume the covers are finite.
    Consequently, by first inducting over $f^{-1}(U_i)=\cup_{j} V_j$ and using \Cref{cor:comp_bound_local_schemes} yields that $f^{-1}(U_i) \to U_i$ is (strongly) compactly bounded.
    Thus, inducting over $Y=\cup_{i} U_i$ and again using \Cref{cor:comp_bound_local_schemes}  gives us what we want.   
\end{proof}

\begin{lemma}\label{lem:std_smooth_comp_bound}
    A standard smooth morphism of affine schemes (in the sense of \cite[\href{https://stacks.math.columbia.edu/tag/01V5}{Tag 01V5}]{stacks-project}) is compactly bounded.
\end{lemma}

\begin{proof}
    A standard smooth morphism of affine schemes corresponds, by definition, to a ring morphism $R \to S \colonequals R[x_1, \ldots, x_n]/(f_1, \ldots, f_c)$ with $\det (\partial f_i/\partial x_j)_{i,j}$ invertible in $S$.

    First, assume $R$ is Noetherian.
    Let $\mathfrak{p}$ be a prime of $R[x_1, \dots, x_n]$.
    As $S$ is a relative global complete intersection over $R$ \cite[\href{https://stacks.math.columbia.edu/tag/00T7}{Tag 00T7}]{stacks-project}, it follows by \cite[\href{https://stacks.math.columbia.edu/tag/00SV}{Tag 00SV}]{stacks-project} that $f_1, \ldots, f_c$ is a regular sequence in $R[x_1, \ldots , x_ n]_{\mathfrak p}$.
    Consequently, $S_{\mathfrak p}$ is a perfect complex over $R[x_1, \ldots, x_n]_{\mathfrak p}$.
    Thus, by \cite[Theorem 4.1]{Avramov/Iyengar/Lipman:2010}, $S$ is a perfect complex over $R[x_1, \ldots, x_n]$.
    
    Next, by \cite[\href{https://stacks.math.columbia.edu/tag/00SU}{Tag 00SU}]{stacks-project}, there exists some finite type $\mathbb{Z}$-subalgebra $R^\prime\subseteq R$ with $f_1,\ldots,f_c\in R^\prime[x_1,\ldots, x_n]$ and $S^\prime:= R^\prime[x_1, \dots x_n]/(f_1, \ldots, f_c)$  a relative global complete intersection over $R^\prime$ (in particular, it is flat over $R^\prime$ \cite[\href{https://stacks.math.columbia.edu/tag/00SW}{Tag 00SW}]{stacks-project}).   
    By the previous paragraph $S^\prime$ is perfect over $R^\prime[x_1, \ldots, x_n]$ and consequently $S=S^\prime\otimes_{R^\prime} R$ is perfect over $R[x_1, \ldots ,x_n]$.
    
    Thus $S$ has finite projective dimension as an $R$-module and therefore clearly lived inside some $\overline{\langle R\rangle}_n$.
    As $\operatorname{Perf}(S)=\langle S\rangle$, this shows $\operatorname{Spec}(S)\to \operatorname{Spec}(R)$ is compactly bounded.
\end{proof}

\begin{proposition}\label{prop:smooth_compact_bounded}
    Any smooth morphism of quasi-compact separated schemes is compactly bounded.
\end{proposition}

\begin{proof}
    By \cite[\href{https://stacks.math.columbia.edu/tag/01V7}{Tag 01V7}]{stacks-project} any smooth morphism is locally a standard smooth morphism of affine schemes. 
    Hence, we are done by the combination of \Cref{cor:comp_bound_affine_local} and \Cref{lem:std_smooth_comp_bound}.
\end{proof}

Moreover, for separated schemes of finite type over `small' rings one can show that any flat morphism is compactly bounded. 
For this recall the following result.

\begin{lemma}[{\cite[Seconde Partie, Corollarie 3.3.2]{Raynaud/Gruson:1971} or \cite[Corollary 3.13]{Simson:1977}}]\label{lem:proj_dim_flat} 
The projective dimension of any flat module over a ring of cardinality bounded by $\aleph_{n}$, for some positive integer $n$, is bounded above by $n+1$.
\end{lemma}

\begin{proposition}
    Let $R$ be any ring of cardinality bounded by $\aleph_n$, for some positive integer $n$, and let $Y$ be a separated scheme of finite type over $R$. 
    Then any flat, quasi-compact, and separated morphism $f \colon X \to Y$ is compactly bounded.
\end{proposition}

\begin{proof}
    Immediate from \Cref{cor:comp_bound_affine_local} and \Cref{lem:proj_dim_flat}. 
\end{proof}

\section{Zariski descent}
\label{sec:Zariski}

We start by introducing increasing assumptions that ensure the presheaves of triangulated categories we consider behave suitably well.

\begin{hypothesis*}
    Let $\mathcal{T}$ be  $\mathcal{C}$-presheaf of triangulated categories.
    We say $\mathcal{T}$ satisfies
    \begin{enumerate}[label={(}\alph*{)}]
        \item\label[hypothesis]{hyp:1} if it is compactly generated and pullbacks preserve compact objects (equivalently the pushforwards preserve coproducts),
        \item\label[hypothesis]{hyp:2} if, additionally, it admits a $t$-structure, pullbacks are t-exact and $\mathcal{T}(U)^b_c$ is a thick subcategory for all $U\in\mathcal{C}$,
        \item\label[hypothesis]{hyp:3} if, additionally, $\mathcal{T}(U)^c \subseteq\mathcal{T}(U)^b$, i.e.\ the compacts are bounded with respect to the $t$-structure.
    \end{enumerate}
\end{hypothesis*}

\begin{remark}
    For example, $\mathcal{T}(U)^b_c$ being thick holds whenever there exists a compact generator $G\in\mathcal{T}(U)$ and integer $A>0$ with $\operatorname{Hom}(G[A],\mathcal{T}^{\leq 0} (U)\big )=0$, see \cite[Proposition 0.19]{Neeman:2021c}.
    In particular, this holds when $\mathcal{T}(U)$ is approximable.

    The containment $\mathcal{T}(U)^c \subseteq\mathcal{T}(U)^b$ is equivalent to $\mathcal{T}(U)^c \subseteq\mathcal{T}(U)^b_c$, so one can make sense of their quotient.
    However, this is far from automatic and there are counterexamples.
    For example, in the stable homotopy category of topological spectra the sphere spectrum is a compact generator, but not bounded.
\end{remark}

\begin{example}\label{ex:presheaves_on_Zariski_site_satisfy_abc}
    Continuing \Cref{ex:presheaves_on_Zariski_site}, it is well-known that the $X_{Zar}$-presheaf $\mathcal{T}:=D_{\operatorname{qc}}$ satisfies \Cref{hyp:1,hyp:2,hyp:3}.
    Indeed, equipping $D_{\operatorname{qc}}(-)$ with the standard $t$-structure $(D_{\operatorname{qc}}^{\leq 0} (-), D_{\operatorname{qc}}^{\geq 0}(-))$ (these are indeed subpresheaves by flatness) one has
    \begin{displaymath}
        \mathcal{T}(-)^c=\operatorname{Perf}(-)\quad\text{and}\quad\mathcal{T}(-)^b_c=D^b_{\operatorname{coh}}(-).   
    \end{displaymath}
\end{example}

The hypotheses are necessary to ensure we get well-behaved subpresheaves, as in the algebro-geometric setting, which the following lemma shows.

\begin{lemma}
    Let $\mathcal{T}$ be a $\mathcal{C}$-presheaf. 
    Then $\mathcal{T}^c$, $\mathcal{T}^b$ and $\mathcal{T}^b_c$, defined on objects by respectively 
    \begin{displaymath}
        U\mapsto \mathcal{T}(U)^c,\quad U\mapsto \mathcal{T}(U)^b, \quad\text{and}\quad U\mapsto \mathcal{T}(U)^b_c,
    \end{displaymath}
    are subpresheaf of $\mathcal{T}$ when the latter satisfies the appropriate \ref{hyp:1}, \ref{hyp:2} or \ref{hyp:3}.
\end{lemma}

\begin{proof}
    Let $f\colon X \to Y$ be a morphism in $\mathcal{C}$. 
    By assumption $f^\ast \colon \mathcal{T}(Y) \to \mathcal{T}(X)$ preserves compacts and is $t$-exact.
    Hence, clearly, $\mathcal{T}^c$ and $\mathcal{T}^b$ are subfunctors.
    To check that $\mathcal{T}^b_c$ is a subfunctor, take $F\in \mathcal{T}^b_c (X)$. 
    By definition this means that, for each integer $m>0$, there exist an $E\in\mathcal{T}^c (Y)$ and a morphism $E \to F$ with cone in $\mathcal{T}^{\leq -m}(Y)$. 
    By $t$-exactness it follows that the $\operatorname{cone}(f^\ast E \to f^\ast F)\in\mathcal{T}^{\leq -m}(X)$. As $f^\ast E$ is compact, this shows $f^\ast F\in\mathcal{T}^b_c (U)$.
\end{proof}

A last hypothesis.

\setcounter{HypCounter}{3}
\begin{hypothesis}\label{hyp:4}
    Let $\mathcal{T}$ be a $\mathcal{C}$-presheaf with $t$-structure. 
    We say a morphism $i\colon U\to X$ satisfies \Cref{hyp:4} if $i^\ast\colon \mathcal{T}(X)\to \mathcal{T}(U)$ restricts to essentially dense\footnote{I.e.\ essentially surjective up to direct summands.} functors
    \begin{displaymath}
            \mathcal{T}^b_c(X) \xrightarrow{\  i^\ast \ } \mathcal{T}^b_c (U)\quad\text{and}\quad
            \mathcal{T}^c (X) \xrightarrow{\  i^\ast \ }  \mathcal{T}^c(U).
    \end{displaymath}
\end{hypothesis}

\begin{example}\label{ex:presheaves_on_Zariski_site_satisfies_d}
    Again, \Cref{hyp:4} holds in the setting of \Cref{ex:presheaves_on_Zariski_site}, one even obtains localizations.
    The one for $\operatorname{Perf}$ is well-known and the one for $D^b_{\operatorname{coh}}$ follows by e.g.\ \cite[Lemma 2.2]{Chen:2010} or \cite[Theorem 4.4]{Elagin/Lunts/Schnurer:2020}.
\end{example}

We are almost ready to state and prove the main theorem, but first a few lemmas concerning lifting generation along Verdier quotients.
For these, recall some notions introduced in \cite{Avramov/Buchweitz/Iyengar/Miller:2010}. Let $\mathsf{T}$ be a triangulate category and $\mathsf{S}\subseteq\mathsf{T}$ be subcategory. 
We say $E\in \mathsf{T}$ is \textbf{finitely built} by $\mathsf{S}$ if $E\in\langle \mathsf{S} \rangle$. The \textbf{level of $E$ with respect to $\mathsf{S}$},  denoted $\operatorname{level}^{\mathsf{S}} (E)$, is the smallest $n\geq 0$ for which $E\in\langle \mathsf{S} \rangle_n$.

\begin{lemma}\label{lem:subcategory_to_object_generation}
    Let $\mathsf{T}$ be a triangulated category and $\mathsf{S}\subseteq\mathsf{T}$ be a subcategory. 
    For each $E\in\langle \mathsf{S} \rangle$, there exists an $S\in\langle \mathsf{S} \rangle_1$ with $\operatorname{level}^{\mathsf{S}}(E) = \operatorname{level}^S (E)$.
\end{lemma}

\begin{proof}
    When $E=0$ there is nothing to prove, so we assume from now on that this is not the case.
    As $0\neq E\in\langle\mathsf{S}\rangle$, $n:=\operatorname{level}^{\mathsf{S}}(E)>0$ is finite and we prove the claim by induction on this.

    If $n=1$, i.e.\ $E\in\langle\mathsf{S}\rangle_1$, then $S=E$ satisfies the claim.
    So, suppose the claim holds for all objects of $\operatorname{level}\leq n$ and that $\operatorname{level}^{\mathsf{S}} (E) = n+1$.
    By assumption there exists a distinguished triangle 
    \begin{equation}\label{eq:disttria}
        A \to E \oplus E^\prime \to B \to A[1]
    \end{equation}
    with $A\in\langle \mathsf{S} \rangle_n$ and  $B\in\langle \mathsf{S} \rangle_1$. 
    Hence, the induction hypothesis yields an $S^\prime\in \langle\mathsf{S}\rangle_1$ with $\operatorname{level}^{\mathsf{S}} (A) = \operatorname{level}^{S^\prime} (A)$
    (and note that $\operatorname{level}^{\mathsf{S}} (A)$ must equal $n$ as any smaller value would contradict $\operatorname{level}^{\mathsf{S}} (E)=n+1$). 
    Let $S := S^\prime \oplus B$; clearly $S\in\langle \mathsf{S} \rangle_1$ and the triangle \eqref{eq:disttria} gives us $\operatorname{level}^S (E)\leq n+1$. However, as $\operatorname{level}^{\mathsf{S}} (S)=1$, we also have
    \begin{displaymath}
        \begin{aligned}
            n+1 &= \operatorname{level}^{\mathsf{S}} (E) \\&\leq \operatorname{level}^{\mathsf{S}} (S) \operatorname{level}^S (E) \\&= \operatorname{level}^S (E).
        \end{aligned}
    \end{displaymath}
    That is, the level of $E$ with respect to $S$ is exactly $n+1$ as was needed.
\end{proof}

\begin{lemma}\label{lem:finitely_building_verdier_quotient}
    Let $\mathsf{T}$ be a triangulated category, $\mathsf{K}\subseteq\mathsf{T}$ be a thick triangulated subcategory and denote the Verdier quotient by $q:\mathsf{T}\to\mathsf{T}/\mathsf{K}$.
    If $E$ and $G$ are objects in $\mathsf{T}$ with $q(E)\in\langle q(G) \rangle_n$, for some $n\geq 1$, then $E\in\langle \{G\oplus K\mid K\in \mathsf{K}\} \rangle_{6n-3}$.
\end{lemma}

\begin{proof}
    We prove the claim by induction on $n$. 
    
    Suppose $n=1$, i.e.\ $q(E)\in\langle q(G) \rangle_1$. 
    Then there exists an object $E^\prime\in\mathsf{T}$ and an isomorphism $q(E\oplus E^\prime) = \oplus_{i\in I} q(G) [s_i]$ (for some finite set $I$). 
    This isomorphism corresponds to a roof 
    \begin{displaymath}
        E\oplus E^\prime \xleftarrow{\quad f\quad} F\xrightarrow{\quad g\quad}\bigoplus_{i\in I} q(G) [s_i]\quad\text{in $\mathsf{T}$},
    \end{displaymath}
    where $\operatorname{cone}(f)$, $\operatorname{cone}(g)\in\mathsf{K}$. 
    Thus, $F$ belongs to $\langle G \oplus \mathsf{K} \rangle_2$, and so $E$ is in $\langle G \oplus \mathsf{K} \rangle_3$.

    Next, suppose the claim holds for all integers $\leq n$ and $q(E)\in\langle q(G) \rangle_{n+1}$. 
    Then, by definition, there exists a distinguished triangle
    \begin{equation}\label{eq:finitely_building_verdier_quotient}
        q(A) \to q(E \oplus E^\prime) \to q(B) \to q(A)[1]
    \end{equation}
    with $q(A)\in\langle q(G) \rangle_n$ and $q(B)\in\langle q(G) \rangle_1$.
    Moreover, there exists a distinguished triangle
    \begin{displaymath}
        A^\prime \to E^{\prime\prime} \to B^\prime \to A^\prime[1]\quad\text{ (in $\mathsf{T}$)}
    \end{displaymath}
    whose image under the quotient functor $q \colon \mathsf{T} \to \mathsf{T}/\mathsf{K}$ is isomorphism to the triangle $\eqref{eq:finitely_building_verdier_quotient}$.
    Consequently, $q(A^\prime)=q(A)\in\langle q(G) \rangle_n$ and $q(B^\prime)=q(B)\in\langle q(G) \rangle_1$ and the induction hypothesis yields $A^\prime\in\langle G \oplus \mathsf{K} \rangle_{6n-3}$ and $B^\prime\in\langle G \oplus \mathsf{K} \rangle_3$. 
    Thus, ${E}^{\prime\prime}\in\langle G \oplus \mathsf{K} \rangle_{6n+1}$ and a similar argument with a roof, corresponding to $q(E\oplus E^\prime)=q(E^{\prime\prime})$, shows $E\in\langle G \oplus \mathsf{K} \rangle_{6n+3}$.
\end{proof}

The following now follows immediately from the previous two lemmas. 

\begin{lemma}\label{lem:finitely_build_a_la_verdier_quotient_refined}
    Let $\mathsf{T}$ be a triangulated category, $\mathsf{K}\subseteq\mathsf{T}$ be a thick subcategory and denote the Verdier quotient by $q:\mathsf{T}\to\mathsf{T}/\mathsf{K}$.
    If $E$ and $G$ are objects of $\mathsf{T}$ such that $q(E)$ belongs to $\langle q(G) \rangle_n$, for some $n\geq 1$, then there exists $K\in\langle \mathsf{K} \rangle_1$ with $E\in\langle G\oplus K \rangle_{6n-3}$.
\end{lemma}

We can now finally state and prove the main theorem.

\begin{theorem}\label{thm:generators_presheaf}
    Let $\mathcal{T}$ be $\mathcal{C}$-presheaf of triangulated categories, 
    \begin{displaymath}
        \begin{tikzcd}
            {W} & V \\
            U & X
            \arrow[from=1-1, to=1-2]
            \arrow[from=1-1, to=2-1]
            \arrow["j", from=1-2, to=2-2]
            \arrow["i"', from=2-1, to=2-2]
            \arrow["k"{description}, from=1-1, to=2-2]
        \end{tikzcd}
    \end{displaymath}
    be a Zariski square in $\mathcal{C}$ and suppose either
    \begin{enumerate}
        \item\label{item:generators_presheaf1} \Cref{hyp:1} holds, $\mathcal{S}=\mathcal{T}^c$ and $i_\ast$, $j_\ast$ and $k_\ast$ are compactly bounded,
        \item\label{item:generators_presheaf2} \Cref{hyp:2} holds, $\mathcal{S}=\mathcal{T}^b_c$, $i_\ast$, $j_\ast$ and $k_\ast$ are coherently bounded and $\mathcal{T}^{\geq 0}(X)$ is closed under coproducts,
        \item\label{item:generators_presheaf3} \Cref{hyp:3} holds, $\mathcal{S}=\mathcal{T}^b_c/\mathcal{T}^c$, $i_\ast$, $j_\ast$ and $k_\ast$ are both coherently and strongly\footnote{In fact, the `strongly' is not needed for the claim concerning classical generators below.} compactly bounded and $\mathcal{T}^{\geq 0}(X)$ is closed under coproducts. 
    \end{enumerate}
    If $\mathcal{S}(U)$, $\mathcal{S}(V)$ and $\mathcal{S}(W)$ admit strong (resp.\ classical) generators, then $\mathcal{S}(X)$ admits a strong (resp.\ classical) generator.
    
    Furthermore, suppose the morphisms in the square satisfy \Cref{hyp:4}.
    Then the converse also holds and the condition on $W$ is automatic, i.e.\ $\mathcal{S}(U)$ and $\mathcal{S}(V)$ admit strong (resp.\ classical) generators if and only if $\mathcal{S}(X)$ admits a strong (resp.\ classical) generator.       
\end{theorem}

\begin{proof}
    We prove the claims concerning strong generators, the analogous claims for classical generators follows by allowing $n=\infty$ in the $\langle\dots\rangle_n$'s below.
    Moreover, we only prove case \eqref{item:generators_presheaf3} as the other two are similar.

    Thus, let $\mathcal{S}=\mathcal{T}^b_c/\mathcal{T}^c$ and let $q_?\colon \mathcal{T}^b_c(?)\to \mathcal{S}(?)$ denote the various quotient functor.
    Suppose $\mathcal{S}(U)$, $\mathcal{S}(V)$ and $\mathcal{S}(V)$ admit respective strong generators.
    To show $\mathcal{S}(X)$ is strongly generated, it is equivalent to show there exists a $G\in\mathcal{T}^b_c(X)$ and integer $l$ with $\langle G \oplus \mathcal{T}^c(X) \rangle_l = \mathcal{T}^b_c(X)$.

    Let $E\in\mathcal{T}^b_c(X)$ be arbitrary. 
    Denote the strong generator of $\mathcal{S}(U)$ by $q_U(G_U)$ and suppose it generates with $n^\prime-1$ cones.
    As $q_U(i^\ast E)\in \langle q_U(G_U)\rangle_{n^\prime}$, it follows by \Cref{lem:finitely_build_a_la_verdier_quotient_refined} that there exists $C_U\in\mathcal{T}^c(U)$ with $i^\ast E \in \langle G_U \oplus C_U \rangle_n$ where $n= 6n^\prime-3$.
    Next, note that by the boundedness assumptions we can find a $G\in\mathcal{T}^b_c(X)$, $C\in \mathcal{T}^c(X)$ and an integer $m$ (that does not depend on $E$!) such that
    \begin{displaymath}
        i_\ast G_U\in\overline{\langle G\rangle}_m\quad\text{and}\quad
        i_\ast C_U\in\overline{\langle C\rangle}_m.
    \end{displaymath}
    It follows that 
    \begin{displaymath}
        i_\ast i^\ast E \in \langle i_\ast G_U \oplus i_\ast C_U \rangle_n\subseteq \overline{\langle G \oplus C \rangle}_{nm}.
    \end{displaymath}
    Furthermore, adding summands to $G$, $C$ and increasing the integers if necessary, we may assume the same holds for $j$ and $k$.
    To finish, it suffices to observe that, by \Cref{lem:MVtriangle}, there is a distinguished triangle
    \begin{displaymath}
        E \to i_\ast i^\ast E \oplus  j_\ast j^\ast E \to k_\ast k^\ast E\to E [1].
    \end{displaymath}    
    Hence, with $l=2nm$ independent of $E$, we have $E\in\overline{\langle G \oplus C \rangle}_l$.
    It follows, by e.g.\ \cite[Lemma 2.14]{DeDeyn/Lank/ManaliRahul:2024a}, that $E\in\langle G\oplus C\rangle_{l}\subseteq \langle G \oplus \mathcal{T}^c(X) \rangle_l$.

    For the converse direction assume \Cref{hyp:4} holds.
    Clearly, the induced functor $i^\ast \colon \mathcal{S}(X) \to \mathcal{S}(U)$ is also essentially dense.
    So, if $G$ is a strong generator for $\mathcal{S}(X)$, then $i^\ast G$ is a strong generator for $\mathcal{S}(U)$.
    A similar reasoning holds for $j$. 
    Additionally, this also implies that $\mathcal{S}(W)$ automatically admits a strong generator whenever either $\mathcal{S}(U)$ or $\mathcal{S}(V)$ do.
\end{proof}

\begin{remark}\label{rmk:extending_to_non_zariski}
    It is worthwhile to observe that in the first part of \Cref{thm:generators_presheaf} one can replace the condition of the square being Zarsiki, by the condition that it induces a triangle as in \Cref{lem:MVtriangle}. 
    Such triangles exist more generally; e.g.\ for any Mayer--Vietoris square as in \cite{Hall/Rydh:2017}, see \cite[Proposition 5.9 (1)]{Hall/Rydh:2017}; in particular, this holds for `\'{e}tale neighborhoods' \cite[Example 5.6]{Hall/Rydh:2017}.
\end{remark}

\begin{remark}\label{rmk:square_thm_weakening}
    Of course, \Cref{thm:generators_presheaf} holds more generally for subpresheaves of $\mathcal{T}$ different from $\mathcal{T}^c$ and $\mathcal{T}^b_c$, granted the necessary assumptions are in place so that the proof goes through.
\end{remark}

The following, to the best of our knowledge, is new (e.g. tells us existence of strong generation for the singularity category is Zariski local).

\begin{corollary}\label{cor:descent_Dsg}
    Let $X$ is a separated Noetherian scheme. Then $D_{\operatorname{sg}}(X)$ admits a strong (resp.\ classical) generator if and only if there is an open cover $X=\cup_i U_i$ with each $D_{\operatorname{sg}}(U_i)$ admitting a strong (resp.\ classical) generator.
\end{corollary}

\begin{proof}
    Apply the theorem, by inducting over a finite subcover, to the $X_{Zar}$-presheaf $\mathcal{T}=D_{\operatorname{qc}}$ taking note of $\mathcal{T}^b_c/\mathcal{T}^c=D_{\operatorname{sg}}$, \Cref{ex:presheaves_on_Zariski_site_satisfy_abc,ex:presheaves_on_Zariski_site_satisfies_d}.
\end{proof}

Moreover, \Cref{thm:generators_presheaf} allows us to recover the previous results in the literature that were proven along the same lines.

\begin{corollary}[{\cite[Theorem D]{Lank:2024}}]\label{cor:lank2024}
   Let $X$ be a separated Noetherian scheme. Then $D^b_{\operatorname{coh}}(X)$ admits a strong generator if and only if $D^b_{\operatorname{coh}}(U)$ admits a strong generator for each affine open $U\subseteq X$.
\end{corollary}

\begin{remark}\label{rmk:lank2024classical}
    One can, of course, replace strong generator by classical generator in \Cref{cor:lank2024}.
\end{remark}

\begin{corollary}[{\cite[Theorem 0.5]{Neeman:2021b} \& \cite[Corollary 7]{Stevenson:2025}}]\label{cor:Neeman2021}
    Let $X$ be a quasi-compact separated scheme. 
    Then $\operatorname{Perf}(X)$ admits a strong generator if and only if $X$  can be covered by open affine subschemes $\operatorname{Spec}(R_i)$ with each $\operatorname{Perf}(R_i)$ admitting a strong generator.
    When $X$ is coherent (i.e. $\mathcal{O}_X$ is a coherent sheaf), the latter is equivalent with  each $R_i$ being of finite weak global dimension.
    Additionally, when $X$ is Noetherian this is equivalent with each $R_i$ being of finite global dimension.
\end{corollary}

\begin{proof}
    The first statement is clear.
    When $R$ is coherent $\operatorname{Perf}(R)$ admitting a strong generator is equivalent to $R$ having finite weak global dimension by \cite[Theorem 6]{Stevenson:2025}.
    Additionally, when $R$ is Noetherian $\operatorname{Perf}(R)$ admitting a strong generator is equivalent to $R$ having finite global dimension by \cite[Corollary 4.3.13]{Letz:2020} or \cite[Proposition 10]{Krause:2024}.
    Alternatively, it is well-known that weak global dimension equals global dimension for Noetherian rings.
\end{proof}

\begin{remark}
    Using \Cref{cor:Neeman2021} and (proof of) \Cref{prop:smooth_compact_bounded} one can show for any smooth covering $\{X_i\to X\}$ of a quasi-compact separated scheme $X$, $\operatorname{Perf}(X)$ admits a strong generator if and only if every $\operatorname{Perf}(X_i)$ does.
    Of course, as being regular (:=locally Noetherian with regular stalks) is smooth local, one does not gain a lot.
\end{remark}

Lastly, it is worthwhile to note that \Cref{thm:generators_presheaf} also applies in a mild noncommutative set-up and recovers \cite[Theorem B]{DeDeyn/Lank/ManaliRahul:2024a} and \cite[Proposition 5.6]{DeDeyn/Lank/ManaliRahul:2024b}.

\section{Variations on a theme}
\label{sec:variations}

\subsection{Big versions}
\label{subsec:big}

Recall that a strong $\oplus$-generator (or a `big generator') in a triangulated category $\mathsf{T}$ admitting small coproducts is an object $G\in\mathsf{T}$ satisfying $\mathsf{T}=\overline{\langle G \rangle}_n$ for some $n\geq 0$. 
There are straightforward generalizations of \Cref{thm:generators_presheaf} \eqref{item:generators_presheaf1} and \eqref{item:generators_presheaf2} to the big setting, where one changes `$\mathcal{S}(X)$ admits a strong generator' to `$\mathcal{T}(X)$ admits strong $\oplus$-generator from $\mathcal{S}(X)$', etc.
Moreover, the added flexibility of working on the big level allows one to remove some hypotheses in the statement.

\begin{theorem}\label{thm:big_generators_presheaf}
    Let $\mathcal{T}$ be $\mathcal{C}$-presheaf of triangulated categories, 
    \begin{displaymath}
        \begin{tikzcd}
            {W} & V \\
            U & X
            \arrow[from=1-1, to=1-2]
            \arrow[from=1-1, to=2-1]
            \arrow["j", from=1-2, to=2-2]
            \arrow["i"', from=2-1, to=2-2]
            \arrow["k"{description}, from=1-1, to=2-2]
        \end{tikzcd}
    \end{displaymath}
    be a Zariski square in $\mathcal{C}$ and suppose either
    \begin{enumerate}
        \item\label{item:big_generators_presheaf1} \Cref{hyp:1} holds, $\mathcal{S}=\mathcal{T}^c$ and $i_\ast$, $j_\ast$ and $k_\ast$ are compactly bounded,
        \item\label{item:big_generators_presheaf2} \Cref{hyp:2} holds, $\mathcal{S}=\mathcal{T}^b_c$, $i_\ast$, $j_\ast$ and $k_\ast$ are coherently bounded,
    \end{enumerate}
    Then $\mathcal{T}(U)$ and $\mathcal{T}(V)$ admit strong $\oplus$-generators from respectively $\mathcal{S}(U)$ and $\mathcal{S}(V)$ if and only if $\mathcal{T}(X)$ admits a strong $\oplus$-generator from $\mathcal{S}(X)$.
\end{theorem}

\begin{proof}
    Copy and suitably adjust the proof of \Cref{thm:generators_presheaf}.
    Note that the localization of \Cref{lem:preflatmono_are_loc} (really just the essential surjectivity of the pullback) gives that, e.g., $\mathcal{T}(U)$ admitting a strong $\oplus$-generator from $\mathcal{S}(U)$ implies that $\mathcal{T}(W)$ admits a strong $\oplus$-generator from $\mathcal{S}(W)$ (hence \Cref{hyp:4} is unnecessary).
\end{proof}

In particular, this recovers \cite[Theorem A]{DeDeyn/Lank/ManaliRahul:2024a} and \cite[Corollary 5.7]{DeDeyn/Lank/ManaliRahul:2024b}.

\subsection{Beyond Zariski}
\label{subsec:topologies}

There are ways of dealing with strong generation for more exotic topologies than the Zariski topology (in the algebro-geometric setting at least where `more exotic topology' makes sense), as was done in \cite{Aoki:2021,DeDeyn/Lank/ManaliRahul:2024b}.
We briefly sketch this approach here, the following is the crucial definition (we state it in the language of tensor triangular geometry, see \cite[Definition 3.18]{Mathew:2016} for the `higher' version).
Moreover, observe that the definition makes sense for any `big' tensor triangulated category, i.e.\ those admitting arbitrary coproducts.
From now on, `tt' stands for either tensor triangulated or tensor triangular.

\begin{definition}[{\cite[Proposition 3.15 \& Definition 3.16]{Balmer:2016}}]\label{def:descendable_object}
    A commutative monoid $A$ in a tt-category $(\mathsf{T},\otimes,\mathbf{1})$ is \textbf{nil-faithful}, or \textbf{descendable}, if the smallest thick tensor-ideal containing $A$, denoted $\operatorname{thick}^\otimes(A)$, equals $\mathsf{T}$.
\end{definition}

The article \cite{Balmer:2016} works in the `small' setting (so with $\mathcal{T}$ essentially small), with \emph{separable} monoids and a slightly stronger notion of triangulated category; however, these are not needed for Proposition 3.15 in loc.\ cit.
(A reason for the latter two assumptions in loc.\ cit.\ is to make sense of the \emph{triangulated} category of $A$-modules, which we do not need). 
For completeness we give those characterizations from Proposition 3.15 in loc.\ cit.\ that we require.

\begin{lemma}\label{lem:descendable_object}
    Let $A$ and $\mathsf{T}$ be as in \Cref{def:descendable_object}.
    The following are equivalent
    \begin{enumerate}
        \item\label{item:des1} $A$ is descendable, i.e.\ $\operatorname{thick}^\otimes(A)=\mathsf{T}$,
        \item\label{item:des2} in some (and hence in every) distinguished triangle 
        \begin{displaymath}
            J\xrightarrow{\quad \xi\quad} \mathbf{1} \xrightarrow{\quad \eta\quad} A \xrightarrow{\quad \zeta \quad} J[1],
        \end{displaymath}
        where $\eta$ is the unit of the monoid $A$, the morphism $\xi$ is $\otimes$-nilpotent, i.e.\ there exists some $n\geq 0$ with $\xi^{\otimes n}=0$,
        \item\label{item:des3} if $f$ is a morphism in $\mathsf{T}$ and $A\otimes f=0$, then $f$ is $\otimes$-nilpotent.
    \end{enumerate}
\end{lemma}

\begin{proof}
    The required parts of \cite[Proposition 3.15]{Balmer:2016} work verbatim (and importantly, essential smallness, separability of $A$ or the stronger notion of triangulated are not needed). 
\end{proof}

The main algebro-geometric example is the following.

\begin{example}\label{ex:hcover}
    Recall that a morphism $f\colon Y\to X$ between Noetherian schemes is called an \textbf{h cover} if it is of finite type and every base change is submersive, i.e.\ surjective and $U \subseteq X$ is open or closed if and only if $f^{-1}(U )\subseteq Y$ is so. 
    In particular fppf, and so in particular smooth and \'etale, covers are h covers.
    By \cite[Proposition 11.25]{Bhatt/Scholze:2017}, for such a morphism $f$, the direct image $\mathbb{R}f_\ast \mathcal{O}_Y$ is descendable when viewed as a commutative monoid in $D_{\operatorname{qc}}(X)$.
\end{example}

The way one can leverage these descendable objects is through the following two results, originally due to Aoki \cite[Proposition 4.4 and Corollary 4.5]{Aoki:2021}.
We simply sketch the arguments, in tt-language, for convenience of the reader.
A `relative' version, for a tt-category acting on some other triangulated category, was used in \cite{DeDeyn/Lank/ManaliRahul:2024b}.

\begin{proposition}\label{prop:aoki_descent}
    Let $A$ and $\mathsf{T}$ be as in \Cref{def:descendable_object}.
    Suppose $\mathsf{S}\subseteq\mathsf{T}$ is a thick triangulated subcategory closed under the action of $A$.
    Then there is an integer $n\geq 0$ such that
    \begin{displaymath}
        \mathsf{S}= \langle\{ A \otimes s\mid s\in \mathsf{S} \} \rangle_n.
    \end{displaymath}
\end{proposition}

\begin{proof}[Sketch of Proof]
    Let $S:=\{ A \otimes s\mid s\in \mathsf{S} \}$, one shows by induction, for any $s\in \mathsf{S}$ and integer $i$, that $\operatorname{cone}(\xi^{\otimes i})\otimes s \in \langle S\rangle_i$ (where $\xi$ is defined as in \Cref{lem:descendable_object}).
    Hence, if $n$ is the smallest integer for which $\xi^{\otimes n}=0$, we see that $\operatorname{cone}(\xi^{\otimes n})=\operatorname{cone}(0)=\mathbf{1}\oplus J^{\otimes n}$.
    Consequently, $s\in \langle S\rangle_n$.
\end{proof}

\begin{corollary}
    Let $f^\ast\colon\mathsf{S}\to\mathsf{T}$ be a strongly monoidal functor between tt-categories admitting a right adjoint $f_\ast$ with $f^\ast\dashv f_\ast$ satisfying the projection formula, e.g.\ $\mathsf{S}$ and $\mathsf{T}$ are both rigidly-compactly generated tt-categories.
    Suppose $\mathsf{S}=\overline{\langle S\rangle}_r$, for some integer $r$ and subset $S\subseteq \mathsf{S}$, and $f_\ast(\mathbf{1}_{\mathsf{S}})$ is descendable.
    Then $\mathsf{T}=\overline{\langle f_\ast S\rangle}_m$ for some integer $m$.
\end{corollary}

\begin{proof}[Sketch of Proof]
    By the projection formula, for $t\in\mathsf{T}$, $f_\ast\mathbf{1}_S\otimes t = f_\ast(\mathbf{1}_S\otimes f^\ast t)=f_\ast f^\ast t$.
    Hence, with $n$ as in \Cref{prop:aoki_descent}, $\mathsf{T}= \langle f_\ast\mathbf{1}_S \otimes \mathsf{T} \rangle_n = \langle f_\ast \mathsf{S} \rangle_n = \overline{\langle S \rangle}_{nr}$.
\end{proof}

\subsection{Re:Stacks}
\label{subsec:stacks}

This section briefly discuss some applications to algebraic stacks. 
Our conventions for algebraic stacks follow those of \cite{stacks-project} unless explicitly mentioned otherwise. 
In particular, representable means `representable by schemes' and we always say `representable by algebraic spaces' if so. 
We use usual letters ($U$, $X$, $Y$, etc.) to refer to schemes and algebraic spaces, and curly ones ($\mathcal U$, $\mathcal X$, $\mathcal Y$, etc.) to refer to algebraic stacks. 
In addition, we refer to \cite{Hall/Rydh:2017} for background on derived categories and functors of algebraic stacks; in particular, we follow them and define, for a morphism of algebraic stacks $f\colon \mathcal{X}\to \mathcal{Y}$, the derived pushforward $\mathbb{R}f_\ast:=\mathbb{R}(f_{\operatorname{qc}})_\ast\colon D_{\operatorname{qc}}(\mathcal{X}) \to D_{\operatorname{qc}}(\mathcal{Y})$ as the right adjoint to the pullback functor $\mathbb{L}f^\ast:=\mathbb{L}(f_{\operatorname{qc}})^\ast\colon D_{\operatorname{qc}}(\mathcal{Y})\to D_{\operatorname{qc}}(\mathcal{X})$ (which is subtle). 

A first useful fact is the following.

\begin{lemma}\label{lem:stacks_representable}
    Let $f\colon \mathcal{Z} \to \mathcal{Y}$ and $g\colon \mathcal{Y} \to \mathcal{X}$ be morphism of algebraic stacks. 
    If $g\circ f$ is representable by algebraic spaces, then so is $f$.
\end{lemma}

\begin{proof}
    For example, this follows from \cite[\href{https://stacks.math.columbia.edu/tag/04Y5}{Tag 04Y5}]{stacks-project} using the fact that for two functions $\phi$ and $\theta$, if $\phi\circ \theta$ is injective, then so is $\theta$. 
\end{proof}

Next, recall from \cite[Definition 2.4]{Hall/Rydh:2017} that a \textbf{concentrated morphism} of algebraic stacks if one that is quasi-compact quasi-separated and for which the derived pushforward of any base change (from a quasi-compact quasi-separated algebraic stack) has finite cohomological dimension; examples are morphisms representable by algebraic spaces (see \cite[Lemma 2.5]{Hall/Rydh:2017}). 
Their derived pushforwards (on the derived category of complexes with quasi-coherent cohomology) are especially well-behaved \cite[Theorem 2.6]{Hall/Rydh:2017}.  
In particular, one says a (quasi-compact quasi-separated) algebraic stack $\mathcal{X}$ is \textbf{concentrated} when the structure morphism $\mathcal{X} \to \operatorname{Spec}(\mathbb{Z})$ is such.
It is worthwhile to note that any tame algebraic stack is concentrated \cite[Theorem 2.1]{Hall/Rydh:2015}.

The following shows that, with some finiteness assumptions, the derived pushforwards of morphisms of algebraic stacks with concentrated target are always coherently bounded.

\begin{lemma}\label{lem:coherent_boundedness_stacks}
    Let $f\colon \mathcal{Y} \to \mathcal{X}$ be a morphism of finite type between separated Noetherian algebraic stacks where $\mathcal{Y}$ is concentrated. Then 
    $\mathbb{R}f_\ast\colon D_{\operatorname{qc}}(\mathcal{Y})\to D_{\operatorname{qc}}(\mathcal{X})$ is coherently bounded.
\end{lemma}

\begin{proof}
    By \cite[Theorem 5.4]{Rydh:2023} there exists a proper surjective morphism $g\colon X \to \mathcal{X}$ from a scheme $X$ that must be both Noetherian and separated (as $f$ is proper, it is in particular finite type and separated).
    Again applying loc.\ cit.\ to $X\times_\mathcal{X} \mathcal{Y}$ and composing with the projection $X\times_\mathcal{X} \mathcal{Y}\to \mathcal{Y}$ we obtain a proper surjective morphism $h\colon Y \to \mathcal{Y}$ where again $Y$ is necessarily a separated Noetherian scheme. 
    Moreover, there is an induced morphism $f^\prime\colon Y\to X$ that is of finite type.
    In short, we obtain the following (2-)commutative diagram 
    \begin{displaymath}
        \begin{tikzcd}[ampersand replacement=\&]
            Y \& {X\times_{\mathcal{X}}\mathcal{Y}} \& {\mathcal{Y}} \\
        	\& X \& {\mathcal{X}}
        	\arrow[from=1-1, to=1-2]
        	\arrow["h", bend left, from=1-1, to=1-3]
        	\arrow["{f^\prime}"', from=1-1, to=2-2]
        	\arrow[from=1-2, to=1-3]
        	\arrow[from=1-2, to=2-2]
        	\arrow["\lrcorner"{anchor=center, pos=0.125}, draw=none, from=1-2, to=2-3]
        	\arrow["f", from=1-3, to=2-3]
        	\arrow["g", from=2-2, to=2-3]
        \end{tikzcd}
    \end{displaymath}
    Invoking \cite[Theorem 6.2]{Hall/Lamarche/Lank/Peng:2025} one has
    $D^b_{\operatorname{coh}}(\mathcal{Y})=\langle \mathbb{R} h_\ast D^b_{\operatorname{coh}}(Y) \rangle$. 
    Thus, for $E \in D^b_{\operatorname{coh}}(\mathcal{Y})$ arbitrary, there exists an $E_Y\in D^b_{\operatorname{coh}}(Y)$ and integer $n\geq 0$ with $E\in \langle \mathbb{R}  h_\ast E_Y \rangle_n$ (e.g.\
    see \Cref{lem:subcategory_to_object_generation}).
    Additionally, as morphisms of finite type between separated Noetherian schemes are coherently bounded \cite[Theorem 3.5]{Aoki:2021}, we know there exists an object $E_X\in D^b_{coh}(X)$ and an integer $m\geq 0$ with $\mathbb{R}f^\prime_\ast E_Y\in\overline{\langle E_X \rangle}_m$. 
    Consequently, $\mathbb{R} (g \circ f^\prime)_\ast E_Y\in\overline{\langle \mathbb{R} g_\ast E_X \rangle}_m$, and hence, $\mathbb{R}f_\ast E\in\overline{\langle \mathbb{R} g_\ast E_X \rangle}_{nm}$ which shows the claim as $\mathbb{R} g_\ast E_X$ is bounded and coherent.
\end{proof}

The following theorem shows that one can check strong generation for (separated Noetherian concentrated) `quasi-DM' stacks (DM for Deligne--Mumford) on a cover.
By \cite[\href{https://stacks.math.columbia.edu/tag/06MF}{Tag 06MF}]{stacks-project} these are exactly the (separated Noetherian concentrated) algebraic stacks which allow for a cover such as in the statement of the theorem. 

\begin{theorem}\label{thm:quasi-DM_strong_generation}
    Let $\mathcal{S}$ be a separated Noetherian concentrated algebraic stack.
    Assume there exists a surjective, separated, finitely presented, quasi-finite flat morphism $U\to \mathcal{S}$ from a scheme such that $D^b_{\operatorname{coh}}(U)$ admits a strong (resp.\ classical) generator. 
    Then $D^b_{\operatorname{coh}}(\mathcal{S})$ admits a strong (resp.\ classical) generator.
\end{theorem}

\begin{proof}[Proof of \Cref{thm:quasi-DM_strong_generation}]
    Again, we only show strong generation as again classical generation follows by taking $n=\infty$.
    The proof proceeds using \'{e}tale/quasi-finite flat d\'{e}vissage \cite[Theorem D']{Hall/Rydh:2018}.
    To this end, let $\mathbb{E}:=\operatorname{\textbf{Stack}}_{\textrm{repr}, \textrm{sep},\textrm{fp},\textrm{qff}/\mathcal{S}}$ denote the strictly full $2$-category of algebraic stacks over $\mathcal{S}$ consisting of those $\mathcal{X} \to \mathcal{S}$ which are representable by algebraic spaces, separated, finitely presented, quasi-finite flat; we simply denote the objects by $\mathcal{X}$ instead of $\mathcal{X} \to \mathcal{S}$.
    By \Cref{lem:stacks_representable} any morphism in $\mathbb{E}$ is representable by algebraic spaces, and hence, concentrated by \cite[Lemma 2.5]{Hall/Rydh:2017}. 
    In particular, as $\mathcal{S}$ is concentrated so is $\mathcal{X}$ for any $\mathcal{X}\in\mathbb{E}$.
    Let $\mathbb{D}$ denote those $\mathcal{X}\in \mathbb{E}$ such that for any $(\mathcal{Y}\to \mathcal{X})\in \mathbb{E}$ \'{e}tale, $D^b_{\operatorname{coh}}(\mathcal{Y})$ admits a strong generator.

    We first show $U\in \mathbb{D}$, so let $(\mathcal{Y}\to U)\in \mathbb{E}$ be \'{e}tale. 
    As any morphism in $\mathbb{E}$ is quasi-compact, separated and representable by algebraic spaces, $\mathcal{Y}$ is an algebraic space. 
    Moreover, as \'{e}tale morphisms are quasi-finite it follows from Stein factorization/Zariski's Main Theorem \cite[\href{https://stacks.math.columbia.edu/tag/082K}{Tag 082K}]{stacks-project} that $\mathcal{Y}$ is in fact a (separated Noetherian) scheme.
    Thus by \cite[Theorem 1.2]{Lank/Olander:2024} it follows that $D^b_{\operatorname{coh}}(\mathcal{Y})$ admits a strong generator.

    In order to apply \cite[Theorem D']{Hall/Rydh:2018} we need to show the following three conditions:
    \begin{enumerate}
        \item \label{item:strong_gen_cover_etale1} If $(\mathcal{X}^\prime \to \mathcal{X})\in\mathbb{E}$ is \'{e}tale and $\mathcal{X}\in\mathbb{D}$, then $\mathcal{X}^\prime\in\mathbb{D}$,
        \item \label{item:strong_gen_cover_etale2} if $(\mathcal{X}^\prime\to \mathcal{X})\in\mathbb{E}$ is finite and surjective and $\mathcal{X}^\prime\in\mathbb{D}$, then $\mathcal{X}\in\mathbb{D}$, and
        \item \label{item:strong_gen_cover_etale3} if $(\mathcal{U} \xrightarrow{j} \mathcal{X})$, $(\mathcal{X}^\prime \xrightarrow{f} \mathcal{X})\in\mathbb{E}$ where $j$ is an open immersion and $f$ is \'{e}tale and an isomorphism over the closed reduced substack corresponding to $|\mathcal{X}|\setminus |\mathcal{U}|$, then $\mathcal{X}\in\mathbb{D}$ whenever  where $\mathcal{U}$, $\mathcal{X}^\prime\in\mathbb{D}$.
    \end{enumerate}
    Then by loc.\ cit.\ $\mathcal{S}\in\mathbb{D}$ as $V\in\mathbb{D}$, and so $D^b_{\operatorname{coh}}(\mathcal{S})$ admits a strong generator.

    It rests to show \eqref{item:strong_gen_cover_etale1} to \eqref{item:strong_gen_cover_etale3}.
    By construction \eqref{item:strong_gen_cover_etale1} holds.
    To show \eqref{item:strong_gen_cover_etale2}, suppose $(\mathcal{Y}\to \mathcal{X})\in \mathbb{E}$ is \'{e}tale.
    Applying \cite[Theorem 6.2]{Hall/Lamarche/Lank/Peng:2025} to the base change $\mathcal{Y}^\prime := \mathcal{Y}\times_{\mathcal{X}}\mathcal{X}^\prime \to \mathcal{Y}$ 
    (which is also finite and surjective) and using that $D^b_{\operatorname{coh}}(\mathcal{Y}^\prime)$ admits a strong generator by assumption yields that $D^b_{\operatorname{coh}}(\mathcal{Y})$ has so too.
    Lastly, to show \eqref{item:strong_gen_cover_etale3} pick an \'{e}tale $(\mathcal{Y}\to \mathcal{X})\in \mathbb{E}$.
    By base changing $j\colon\mathcal{U} \to \mathcal{X}$ and  $f\colon\mathcal{X}^\prime \to\mathcal{X}$ along this morphism we may assume that $\mathcal{U}$ and $\mathcal{X}^\prime$ admit bounded derived categories with strong generators and we need to show so for $\mathcal{X}$.
    This follows from \Cref{thm:big_generators_presheaf}, noting \Cref{rmk:extending_to_non_zariski} and \cite[Example 5.6]{Hall/Rydh:2017}, and \Cref{lem:coherent_boundedness_stacks}. 
\end{proof}

For Deligne--Mumford stacks one can show that the choice of covering is irrelevant.

\begin{corollary}\label{cor:DM_stacks_existence_of_generation_from_scheme_cover}
    Let $\mathcal{S}$ be a separated Noetherian concentrated Deligne--Mumford stack. 
    Suppose there exists a surjective quasi-compact separated \'{e}tale morphism $U\to \mathcal{S}$ from a scheme such that $D^b_{\operatorname{coh}}(U)$ admits a strong (resp.\ classical) generator. 
    Then for any quasi-compact separated \'{e}tale morphism $\mathcal{X} \to \mathcal{S}$ representable by algebraic spaces, $D^b_{\operatorname{coh}}(\mathcal{X})$ admits a strong (resp.\ classical) generator.
\end{corollary}

\begin{proof}
    Again we only treat the strong generator case.
    Note that $U^\prime:=U \times_\mathcal{S} \mathcal{X}$ is a scheme (same argument as in the proof of \Cref{thm:quasi-DM_strong_generation}) and so by \cite[Theorem 1.2]{Lank/Olander:2024} applied to the projection $U^\prime \to U$ we have that $D^b_{\operatorname{coh}}(U^\prime)$ admits a strong generator and so also $D^b_{\operatorname{coh}}(\mathcal{X})$ by applying \Cref{thm:quasi-DM_strong_generation} to $U^\prime\to \mathcal{X}$.
\end{proof}

There are natural examples where  \Cref{thm:quasi-DM_strong_generation} applies, e.g.\ to `nice enough' quotient stacks.
Moreover, we recover \cite[Theorem B.1]{Hall/Priver:2024} (the case of strong generation for separated Noetherian quasi-excellent tame algebraic stack of finite Krull dimension) as tame algebraic stacks are concentrated.
\begin{remark}
    Although we did the above for `small generators' it is of course straightforward to adapt everything to hold for strong $\oplus$-generators with bounded and coherent cohomology.
\end{remark}
 
\bibliographystyle{alpha}
\bibliography{mainbib}

\end{document}